\documentclass[a4paper,11pt]{article}

\usepackage{a4wide}
\usepackage{amssymb,amsmath,amsfonts} \allowdisplaybreaks
\usepackage{graphicx}
\usepackage{hyperref}
\usepackage{color}
\usepackage{amsthm}
\usepackage{enumerate}
\numberwithin{equation}{section}    
\theoremstyle{plain}
\newtheorem{Theorem}{Theorem}[section]

\newtheorem{Lemma}[Theorem]{Lemma}
\theoremstyle{definition}

\newtheorem{Definition}[Theorem]{Definition}

\theoremstyle{remark}
\newtheorem{Remark}[Theorem]{Remark}

\renewcommand{\epsilon}{\varepsilon}
\renewcommand{\phi}{\varphi}

\newcommand{\ato}[2]{\genfrac{}{}{0pt}{2}{#1}{#2}}
\newcommand{\sprod}[2]{\left\langle#1,#2 \right \rangle}
\DeclareMathOperator{\Tr}{Tr}

\DeclareMathOperator{\supp}{supp}

\DeclareMathOperator{\Id}{Id}
\DeclareMathOperator{\Eig}{Eig}
\DeclareMathOperator{\ran}{ran}
\newcommand{\eul}{\mathrm{e}}

\newcommand{\EE}{\mathbb{E}}
\newcommand{\RR}{\mathbb{R}}
\newcommand{\CC}{\mathbb{C}}
\newcommand{\NN}{\mathbb{N}}
\newcommand{\ZZ}{\mathbb{Z}}
\newcommand{\PP}{\mathbb{P}}

\newcommand{\cP}{\mathcal{P}}

\newcommand{\cB}{\mathcal{B}}
\newcommand{\cH}{\mathcal{H}}

\newcommand{\cA}{\mathcal{A}}
\newcommand{\cZ}{\mathcal{Z}}
\newcommand{\cT}{\mathcal{T}}
\newcommand{\cU}{\mathcal{U}}

\newcommand{\hm}[1]{\textbf{*}\leavevmode{\marginpar{\tiny%
$\hbox to 0mm{\hspace*{-0.5mm}$\leftarrow$\hss}%
\vcenter{\vrule depth 0.1mm height 0.1mm width \the\marginparwidth}%
\hbox to 0mm{\hss$\rightarrow$\hspace*{-0.5mm}}$\\\relax\raggedright #1}}}
\title
{Uniform existence of the integrated density of states for randomly weighted Hamiltonians on long-range percolation graphs}
\author
{Slim Ayadi\thanks{University of Tunis El Manar, Departement of Mathematics, Facult\' e des Sciences de Tunis, 2092 Tunis, Tunisia} 
\and Fabian Schwarzenberger\thanks{Technische Universit\"a{}t Chemnitz, Fakult\"a{}t f\"u{}r Mathematik, 09107 Chemnitz, Germany} \and Ivan Veseli\'c$^\dagger$}
\date{\today}

\begin{document}
\maketitle
\begin{abstract}
We consider random Hamiltonians defined on long-range percolation graphs over $\ZZ^d$. The Hamiltonian consists of a randomly weighted Laplacian plus a random potential.  
We prove uniform existence of the integrated density of states and express the IDS using a Pastur-Shubin trace formula.
\end{abstract}

%
\section{Introduction}
In the study of solution properties
and spectral features of random ergodic and periodic operators the integrated 
density of states (IDS), also known as spectral distribution function,
plays and important role. 
While it is by its very nature a much simpler object than the original 
one, i.e. the operator family, it exhibits a number of features which 
turn out to be crucial for the understanding of the spectrum and 
the corresponding eigensolutions. Let us spell out some of these features explicitly:
\begin{enumerate}[(A)]
\item the low energy asymptotics of the IDS, 
\item its local and global continuity properties, and
\item its approximability by finite volume analogues.
\end{enumerate}

In fact, (C) is of relevance for all other questions about the IDS since all methods used to answer/understand them  rely on one stage or another on finite volume approximations. 
Of course, depending on the question which is being considered, the type and quality of the approximation (C) 
will vary.

Let us highlight the intimate relationship between the continuity (B) and approximability (C) properties on an elementary level, to provide 
a motivation for the discussion which follows.
On the one hand, if a sequence of probability measures converges weakly to a measure with 
no atoms, the corresponding distribution functions converge already uniformly. On the other hand, 
uniform convergence of continuous distribution functions obviously implies that the limiting 
measure has no atoms.
Now, there are important classes of random operators where it is known that the 
IDS is not continuous (and other, where continuity is still an open question).
In this setting the question arises whether the approximability (C) in the  $L^\infty$-topology
persists or not. More precisely, it is natural to ask, what are reasonable criteria, 
and which is the mechanism, which ensures that the finite volume approximating distribution
functions converge uniformly to the IDS.

Let us review the history of works addressing the above explained approximability question.
In the simplest setting of finite-hopping range ergodic operators on $\ell^2(\ZZ^d)$ 
the continuity of the IDS was established already in \cite{DelyonS-84}.
In fact, this result can be sharpened under very mild conditions to a quantitative form of continuity, 
namely log-H\"older continuity, c.f.~\cite{CraigS-83a}.  

In the corresponding continuum setting, more specifically for ergodic Schr\"odinger 
operators on $L^2(\RR^d)$, for quite some while the continuity of the IDS was known 
for specific classes of random potentials, in particular for the alloy type or continuum Anderson 
potential. This continuity result is typically established via a Wegner estimate, following the 
exposition in the paper \cite{Wegner-81}. Actually, this paper treats (discrete) Anderson models on the 
lattice, but functional analytics tool allow to transfer the methods to the continuum setting.
In any case, for this technique a regular distribution of the values of the random potential has to be assumed. 
For a survey of results in this direction see, e.g. \cite{Veselic-08} and the references cited there. 
If no assumption of the regularity of the distribution of the potential values is assumed
log-H\"older continuity of the IDS of ergodic Schr\"odinger operators in one dimension was established 
in \cite{CraigS-83b}. For dimensions two and three this result was established only very recently
in \cite{BourgainK-11pre}. 
The case $d\geq 4$ is still open. This is due to the fact that the present state of knowledge 
about quantitative unique continuation properties of eigensolutions of Schr\"odinger operators 
based on Carleman estimates, cf. e.g. \cite{BourgainK-05}, 
is good enough for the IDS-continuity proof in dimensions three and less, 
but not above this threshold.

Once one moves away from the Euclidean setting, continuity is no longer 
a prevailing feature of the IDS. There are two important prototypical examples 
of discrete Laplace and Schr\"odinger operators whose IDS exhibit jumps: 
Hamiltonians on percolation clusters \cite{ChayesCFST-86}, \cite{Veselic-05} 
and quasi-crystal graphs \cite{LenzS-03a,LenzS-06,KlassertLS-03}.
In the former model the discontinuities form even a dense set in the spectrum.
Nevertheless, it has been established that the sequence of normalized finite volume
eigenvalue counting functions converges at every energy and even uniformly along the energy axis
to the IDS, cf.~\cite{Veselic-05,LenzS-06, LenzMV-08,LenzV-09}.

The question whether the approximability (C) holds at any given energy is also of interest 
in a purely geometric setting, where no randomness in involved.
More specifically, the papers 
\cite{Lueck-94,MathaiY-02,MathaiSY-03,DodziukLMSY-03}
study the approximation of Betti numbers by their finite volume analogues.
The usual weak convergence of measures is not sufficient to yield this statement.
The results of \cite{LenzV-09,LenzSV-10, PogorzelskiS-12pre} actually apply under very mild and natural geometric assumptions (namely amenability)
and thus lift the pointwise everywhere convergence results of 
\cite{Lueck-94,MathaiY-02,MathaiSY-03,DodziukLMSY-03}
to uniform convergence.

\medskip

However, all papers mentioned so far concern local operators (in the continuum setting)
respectively finite-hopping range operators (in the discrete setting). Let us spell out the last 
property explicitly: A bounded operator $A$ on $\ell^2(G)$ for some graph $G$ is called of 
finite hopping range $R\in \NN$ if for any $\phi\in\ell^2(G)$ and $x \in G$ with distance larger 
than $R$ to $\supp \phi$ 
\[
 (A\phi)(x)=0
\]

The first result for an ensemble of operators on $\ell^2(G)$ beyond this restriction
was achieved in \cite{Schwarzenberger-12}.
It studies Laplace operators on long-range percolation graphs.
These operators are not of finite hopping range, albeit non-zero ``matrix  elements far off the diagonal''
appear with small probability.
Due to the methods applied in \cite{Schwarzenberger-12}, 
the randomness present in the model had to be of finite local complexity.

In the present paper we continue this line of research. Here we are able to treat models
with real-valued entries, possibly continuously distributed, and with long-range interactions.
As mentioned above, the proof of  \cite{Schwarzenberger-12} does not apply in this setting, 
but a combination of ideas from this paper together with methods from \cite{LenzV-09}
allows one to prove uniform approximability in this more general situation.
Also, we will treat  here randomized versions of adjacency as well as Laplace operators.
The results of \cite{Schwarzenberger-12} apply to the second type of ensembles only.

Question (A) has been addressed for certain long-range percolation models before in \cite{AntunovicV-09}, 
extending previous results
for the usual quantum percolation model on $\ZZ^d$, cf. for example~\cite{KirschM-06, MuellerS-07}.
We will quote for completeness sake the result of \cite{AntunovicV-09} below, 
once we have the necessary notation at disposal.

Let us stress an important feature of long-range percolation Hamiltonians. They provide a simple model interpolating between 
discrete random Schr\"odinger operators and random matrices.
This is of interest, since  the two last mentioned classes of operators have quite different spectral features. 
Thus one is led to ask in which aspects and
regimes, long-range percolation Hamiltonians share features with one or the other of these classes.
To explain the structural difference between discrete random Schr\"odinger operators and random matrices 
let us restrict ourselves for the moment to operators on finite segments of $\ZZ$. 
The arising Schr\"odinger operators  are a special type of finite Jacobi matrices, and are 
in particular tri-diagonal. Random matrices have a full array of random entries. The non-zero entries of a Laplacian
of a typical long-range percolation graph are concentrated mostly near the main diagonal, however no diagonal will consist entirely of zeros.
This is the mentioned interpolating property.
The relation between long-range percolation models and random matrices 
was for instance studied in \cite{Ayadi-09b,Ayadi-09a}.

Let us describe the content of the paper in detail. 
In the next section we state the main result in a concise form, discuss extensions to groups and graphs beyond $\ZZ^d$
and the relation to results concerning the low energy asymptotics of the IDS (A) established for long-range percolation graphs in \cite{AntunovicV-09}.

In Section \ref{s:Setting} we present the long-range percolation model and in particular the underlying probability space. 
Furthermore we define the randomly weighted Hamiltonian $H_\omega$ on the long-range percolation graph 
as a selfadjoint and metrically transitive random operator. 
Depending on the choice of the corresponding parameters, 
the operator in question is either the adjacency operator of a long-range percolation graph, 
a Laplacian, or a Schr\"odinger operator (i.e. a Laplacian plus a random potential), each one with random weights on the edges. 
For a realization $\omega$ the restriction of this operator to a finite box $\Lambda_n\subset\ZZ^d$ gives rise to the eigenvalue-counting function $F_n^\omega$. 
This function encodes the distribution of the spectrum of the restricted operator. In the sequel we 
pursue the question whether and in which sense the limit of the sequence of functions $(F_n^\omega)$ exists. 
In Section \ref{s:Weak}, weak convergence of this sequence is established for almost all realizations using a result of Figotin.
In Section \ref{s:Jumps} this statement is upgraded to uniform convergence along the energy axis. 
This is done by proving that the functions $F_n^\omega$ do not only approximate the limit function at its continuity points, 
but also give an efficient estimate of the size of the jumps at a point of discontinuity. 
The key tool to control the size of the jumps is Theorem \ref{theorem:technical}, formulating the main technical contribution of the present paper.

Let us stress that we give a detailed account of all the main steps of the proof thus making it accessible to non-specialists.

\section{Main result}
\label{s:Result}

Here we formulate the main result of the paper using a minimum of notation needed for this purpose. 
More detailed definitions of the framework can be found in Section \ref{s:Setting}.

Denote by $E:=\{\{x,y\}\subseteq \ZZ^d\mid x,y\in \ZZ^d\}$  the set of all edges (or loops) on $\ZZ^d$ and
by $a_e, b_e, e \in E$ a collection of independent real-valued random variables on a probability space $(\Omega,\cA,\PP)$. 
The variance of the $a_e, e \in E$ is uniformly bounded and any two  $a_{\{x,y\}}$  and $a_{\{x+k,y+k\}}$ are identically distributed for $x,y,k \in \ZZ^d$.
Let $p \in \ell^1(\ZZ^d)$ with $0\leq p(x)=p(-x)\leq 1 $ and $b_{\{x,y\}}$ be Bernoulli distributed with parameter $p(x-y)$.
The kernel, resp.~the matrix entries of the random Hamiltonian $H(\omega)$ are given by:
\begin{equation}
 \label{eq:firstDef}
H_{x,y}(\omega):= H_{x,y}^{\alpha,\beta}(\omega) := \begin{cases}
           a_{\{x,y\}}(\omega) b_{\{x,y\}}(\omega) &\text{ if } x\neq y, \\
   \alpha a_{\{x\}}(\omega)b_{\{x\}}(\omega)-\beta\sum_{z\neq x} a_{\{x,z\}}(\omega)b_{\{x,z\}}(\omega) &\text{ if } x=y,
          \end{cases}
\end{equation}
where $\alpha,\beta\in[0,1]$ are fixed numbers. Selfadjointness properties of $H(\omega)$  are discussed in Section \ref{s:Setting}.
Depending on the values $\alpha$ and $\beta$ we obtain several interesting subclasses: 
randomly weighted Laplacians or adjacency operators, with or without random potentials (on the diagonal), 
cf.~Remark \ref{r:alphabeta}.

For $n\in\NN$ let $\Lambda_n:=([-n,n]\cap\ZZ)^d$,  $H_n(\omega)$ be the restriction of $H(\omega)$ to $\Lambda_n$,
and $  F_n^{\omega}(\lambda)$ the number of eigenvalues of $H_n(\omega)$ not exceeding $\lambda$, counting eigenvalues according to their multiplicity.
Set $F:\RR\to\RR$, 
\begin{equation}\label{eq:defF}
 F(\lambda):=\EE\{ \sprod{E_H((-\infty,\lambda])\delta_0}{\delta_0} \},
\end{equation}
 where $\EE\{\cdot\}$ is the expectation with respect to the measure $\PP$, $\sprod{\cdot}{\cdot}$ is the scalar product in $\ell^2(\ZZ)$ and  $E_{H(\omega)}((-\infty,\lambda])$ is the spectral projector of $H(\omega)$ on the interval $(-\infty,\lambda]$.

\begin{Theorem}\label{theorem:main}
Let $F_n^{\omega}, F$ be the distribution functions given above. 
Then there exists a set $\tilde\Omega\subseteq\Omega$ of full measure such that for all $\omega\in\tilde\Omega$  we have
\[
 \lim_{n\to\infty}\sup_{\lambda\in\RR} \left| \frac{F_n^\omega}{|\Lambda_n|}(\lambda) - F(\lambda) \right|=0.
\]
\end{Theorem}
\begin{Remark}
 The limit of the functions $F_n^\omega/|\Lambda_n|$ for $n\to\infty$ is called the \emph{integrated density of states} (IDS). Theorem \ref{theorem:main} shows that this limit exists in the topology of uniform convergence. 
Moreover, the theorem provides the equality of the IDS with the function $F$ given in \eqref{eq:defF}, 
which is the expectation of an diagonal element of associated spectral  projector. 
Note that by translation invariance we obtain for any finite $\Lambda\subseteq \ZZ^d$ 
\begin{equation}
 F(\lambda)=\EE\{ \sprod{E_H((-\infty,\lambda])\delta_0}{\delta_0} \}
= \frac{1}{|\Lambda|} \EE\Bigl\{\sum_{x\in\Lambda}\sprod{E_H((-\infty,\lambda])\delta_x}{\delta_x}\Bigr\}.
\end{equation}
Thus the IDS, originally obtained as a macroscopic limit, can be identified as an \emph{averged trace per unit volume}. An equality of this type is called \emph{Pastur-Shubin trace formula}. 

\end{Remark}

\subsection*{Extension to more general geometries}

In the presentation of our results we have not 
striven for the  maximal possible generality, 
but rather tried to present proofs as explicitly as possible in their most accessible form.
This concerns in particular the restriction to operators defined on $\ell^2(\ZZ^d)$.
In fact, the presented results carry over to operators on $\ell^2(\Gamma)$, where $\Gamma$ 
may be a much more general graph than the lattice $\ZZ^d$.
The explicit calculations in this case can be found in the submitted Thesis \cite{SchwarzenbergerDiss}.
Here we will only state only the scope of the general results:

 Let $G$ be a finitely generated discrete amenable group, and $\Gamma$ a graph on which $G$ acts freely and cocompactly by translations. Completely analogously as in the $\ZZ^d$ setting one can define long-range percolation on such graphs, as well as random operators introduced in Definition \eqref{eq:firstDef}, respectively \eqref{def:tildeH}.  
For such models the results which we use from \cite{FigotinP-92} can be proven analogously. 
Consequently, Lemma \ref{la:sa} and Theorem \ref{theorem:weak} have their generalized counterparts in this setting.
Relying on the ideas of \cite{LenzV-09} and \cite{Schwarzenberger-12} one can see that it is possible to extend the results of Section \ref{s:Jumps}, as well. Here one needs to apply the pointwise ergodic theorem of Lindenstrauss \cite{Lindenstrauss-01} instead of Theorem \ref{theorem:linde} to finally obtain that Theorem \ref{theorem:main}  holds analogously.

Let us note that the method we use here is in the sense efficient, 
that it does not need any condition beyond amenability, i.e. the existence of a F\o{}lner sequence,  
on the discrete group $G$.
In comparison to this, the method of \cite{Schwarzenberger-12} relies on the following additional 
tiling condition:

It is assumed that there exists a F\o{}lner sequence $(Q_n)_n$ such that for each $n\in\NN$ there is a set $T_n=T_n^{-1}\subseteq G$ with the property that $G$ is the disjoint union of the sets $\{Q_n t\mid t\in T_n\}$.
 This assumption is satisfied for many amenable groups, 
however it is not clear  whether it holds for all of them.

\subsection*{Low energy asymptotics}
In \cite{AntunovicV-09} the authors study low energy asymptotics of the IDS for the long-range percolation model. 
Again, let $p \in \ell^1(\ZZ^d)$ with $0\leq p(x)=p(-x)\leq 1 $ and $b_{\{x,y\}}$ be Bernoulli distributed with parameter $p(x-y)$. We define for each $\omega$ the set $E_\omega:=\{\{x,y\}\subseteq \ZZ^d\mid b_{\{x,y\}}(\omega)=1\}$.
The operator under consideration is the (non-weighted) Laplacian $\Delta_\omega$ of the graph $(\ZZ^d,E_\omega)$, i.e. the operator acting on any $\phi\colon \ZZ^d \to \CC$ with finite support by 
\[
(\Delta_\omega\phi)(x)= \sum_{\ato{y\neq x}{\{x,y\}\in E_\omega}} \left(\phi(x)-\phi(y)\right).
\]
 This equals the operator in \eqref{eq:firstDef} in the case $\alpha=0$, $\beta=1$ and where the random variables $a_e$ are constant $1$, see also \eqref{eq:defH2}

An in physical communities common way to introduce the long-range percolation model is the following:
For each pair of vertices $x,y\in \ZZ^d$ let $J_{x,y}$ be a non-negative real number such that
\begin{itemize}
 \item $J_{x,y}=J_{y,x}$
 \item $J_{x+z,y+z}=J_{x,y}$ for all $z\in \ZZ^d$,
 \item $J:=J_x:= \sum_{y\in \ZZ^d} J_{x,y}$ is finite and independent of $x\in \ZZ^d$.
\end{itemize}
We fix $\beta>0$ and declare an edge $\{x,y\}$ to be an element of $E_\omega$ with probability $1-\eul^{-\beta J_{x,y}}$. This gives the random graph $\Gamma_\omega=(\ZZ^d,E_\omega)$.
Notice that the probability that certain edge is an element of $E_\omega$ is increasing in $\beta$. 
Thus, the subcritical phase, in which all clusters are almost surely finite corresponds to small values of the parameter $\beta$ and the supercritical phase in which there exists almost surely an infinite cluster corresponds to large values of the parameter $\beta$. Just like in the case of the  nearest neighbor percolation model these two phases are separated by a single value of the parameter $\beta$. 
The authors of \cite{AntunovicV-09} define the IDS as in \eqref{eq:defF} and prove that for every subcritical $\beta$ there are constants $c(\beta),d(\beta)>0$ such that for $E>0$ small enough
\[
 \exp\left(-c(\beta) E^{-1/2}\right)\leq F(E)-F(0) \leq \exp\left(-d(\beta) E^{-1/2}\right).
\]
Actually the results of \cite{AntunovicV-09} apply to operators on quasi-transitive graphs $\Gamma$.
The present result is complementary to these observations, as we show that the finite volume approximants do actually converge to this limit function given by the Pastur-Shubin formula. 
Furthermore the combination of both results shows that even the approximating functions will exhibit 
exponential behavior for low energies.

Note that in this paper we introduce the long-range percolation model in another, but equivalent (see \cite{Schwarzenberger-12}) way, via a certain function $p\in\ell^1(\ZZ^d)$, 
see \eqref{eq:def:p}. More background on the models considered in \cite{AntunovicV-09} can be found in the review paper \cite{AntunovicV-08b}.

\section{Setting and first results}
\label{s:Setting}

Let $\Gamma$ be the $\ZZ^d$ lattice and denote by $d:\ZZ^d\times \ZZ^d\to \NN_0$ the graph distance in the lattice or equivalently the $\ell^1$-distance in $\ZZ^d$. With this metric we define the $R$-boundary of a set $\Lambda\subseteq \ZZ^d$ by
\[
 \partial^R\Lambda
:=\{x\in\Lambda\mid d(x,y)\leq R\text{ for some }y\in \ZZ^d\setminus\Lambda\}.
\]

 Furthermore we let $E:=\{\{x,y\}\subseteq \ZZ^d\mid x,y\in \ZZ^d\}$ be the set of all subsets of $\ZZ^d$ containing either one or two elements. The set $E$ can be interpreted as the edge set of the complete undirected graph over $\ZZ^d$, containing loops at each vertex.

The probability space $(\Omega,\cA,\PP)$ is given in the following way. The sample space is $\Omega=\prod_{e\in E} (\RR \times \{0,1\})$ and we denote the elements in $\Omega$ by $\omega=(\omega'_e,\omega''_e)_{e\in E}$.  The appropriate $\sigma$-algebra is $\cA=\bigotimes_{e\in E} (\cB(\RR) \otimes\cP(\{0,1\}))$.
In order to define a measure on this space we fix  some $p\in\ell^1(\ZZ^d)$ with
\begin{equation}\label{eq:def:p}
 0\leq p(x)=p(-x)\leq 1 \quad  \quad (x\in\ZZ^d)
\end{equation}
and for each $z\in \ZZ^d$ some probability measure $\mu_z$ on $\RR$ such that there is $v\in\RR$ with
\begin{equation}\label{eq:def:mu}
\int_{\RR}x^2 d\mu_z(x) \leq v^2 \quad \quad (z\in \ZZ^d).
\end{equation}
We set $\PP:= \bigotimes_{\{x,y\}\in E}(\mu_{x-y} \otimes \nu_{x-y})$ where for each $z\in \ZZ^d$ the measure $\nu_z$ is Bernoulli with parameter $p(z)$.

\begin{Remark}\label{rem:cylinder}
 The $\sigma$-algebra $\cA$ is generated by the cylinder sets $\cZ$, which are given the following way
\[
 \cZ=\left\{Z(A_{e_1},B_{e_1},\dots,A_{e_k},B_{e_k})\mid k\in\NN, e_i\in E, A_{e_i}\in\cB(\RR), B_{e_i}\in \cP(\{0,1 \}) \text{ for }i=1,\dots,k\right\}
\]
where
\[
 Z(A_{e_1},B_{e_1},\dots,A_{e_k},B_{e_k})=\left\{\omega\in\Omega\mid \omega_{e_i}'\in A_{e_i}, \omega_{e_i}''\in B_{e_i}\text{ for }i=1,\dots,k\right\}.
\]
\end{Remark}

Now for each $\omega=(\omega_e',\omega''_e)_{e\in E}$ and $e\in E$ we set $a_e(\omega):= \omega'_e$ and $b_e(\omega):=\omega''_e$.
This procedure gives independent random variables $a_e,b_e$, $e\in E$ satisfying $\PP(a_e\in B)=\mu_e(B)$ as well as $\PP(b_e=1)=\nu_e(\{1\})=p(x-y)$ for arbitrary $e=\{x,y\}\in E$ and $B\in \cB(\RR)$. 
Furthermore, by \eqref{eq:def:mu} we have for each $e\in E$
\begin{equation*}
\EE(|a_e|)
\leq v^2 +1 .
\end{equation*}
These random variables induce for each $\omega\in\Omega$ a graph $\Gamma_\omega=(\ZZ^d,E_\omega)$ with weighted edges. Here $\ZZ^d$ is the vertex set and $E_\omega$ is the subset of $E$, where an edge $e\in E$ is an element of $E_\omega$ if and only if $b_e(\omega)=1$. 
In this case one can think of $a_e(\omega)$ as the weight of the edge $e$. For a subset $\Lambda\subseteq\ZZ^d$ and an element $x\in \ZZ^d$ we write $x\stackrel{\omega}{\sim} \Lambda$ if there exists $y\in \Lambda$ with $\{x,y\}\in E_\omega$.


The following Lemma shows that $\Gamma_\omega$ is almost surely locally finite, i.e. with probability one each vertex is incident to only finitely many edges in $\Gamma_\omega$. 
\begin{Lemma}\label{lemma:locfin}
The graph $\Gamma_\omega$ is locally finite for almost all $\omega\in\Omega$.
\end{Lemma}
\begin{proof}
 Fix an element $x\in\ZZ^d$ and consider the events $A_y:=\{ b_{\{x,y\}}=1\}$, $y\in \ZZ^d$. Then clearly 
\[
\sum_{y\in\ZZ^d}\PP(A_y)=\sum_{y\in\ZZ^d} p(x-y) <\infty,
\]
as $p\in\ell^1(\ZZ^d)$.
Hence, the Borel-Cantelli Lemma gives a set $\Omega_x$ of full measure such that each $\omega\in\Omega_x$ is contained in only finitely many $A_y$, $y\in \ZZ^d$. As $\ZZ^d$ is countable $\tilde \Omega:=\bigcap_{x\in\ZZ^d}\Omega_x$ is a set of full measure as well. Furthermore $\Gamma_\omega$ is locally finite for all $\omega\in\tilde\Omega$.
\end{proof}

Given $\gamma\in\ZZ^d$, let us define translations $T_\gamma:\Omega\to \Omega$ by 
\[
 T_\gamma(\omega)
=T_\gamma((\omega'_e,\omega''_e)_{e\in E})
=(\omega'_{e+\gamma},\omega''_{e+\gamma})_{e\in E}
\]
where for $e=\{g,h\}\in E$ we mean by $e+\gamma$ the element $\{g+\gamma,h+\gamma\}\in E$. For $\gamma\in \ZZ^d$ and $B\in\cA$ we denote the image and the preimage of $B$ under $T_\gamma$ by 
\[T_\gamma(B)=\{T_\gamma(\omega) \in\Omega \mid \omega\in B\}
\quad\text{and}\quad
T_\gamma^{-1}(B)=\{\omega \in\Omega\mid T_\gamma(\omega)\in B \}.
\]
Note that for $B\in \cA$ we have $T_\gamma^{-1}(B)=T_{\gamma^{-1}}(B)$. We further define $T$ to be the mapping $\gamma\mapsto T_\gamma$ which maps each element of $\ZZ^d$ into the space of automorphisms on $(\Omega,\cA,\PP)$. Note that by definition $T$ is ergodic if and only if for any $B\in\cA$ with $T_\gamma(B)=B$ for all $\gamma\in\ZZ^d$ one has $\PP(B)\in\{0,1\}$.
The following result is basic, but we do not know an explicit reference in the literature, so we include a proof for completeness sake.

\begin{Lemma}\label{lemma:Tergodic}
$T$ is a measure preserving, ergodic left-action on $(\Omega,\cA,\PP)$.
\end{Lemma}
\begin{proof}
For an edge $e=\{g,h\}\in E$, vertices $x,y\in\ZZ^d$ and $\omega \in \Omega$ we have $T_0(\omega)=\omega$ and
\begin{align*}
 T_{x+y}(\omega)  
= (\omega'_{e+x+y},\omega''_{e+x+y})_{e\in E} 
= T_x(T_y(\omega))
\end{align*}
which shows that $T$ is a left action of $\ZZ^d$ on $\Omega$.
  
By definition of $\PP$ and the random variables $a_e$ and $b_e$ we have $\PP(a_e \in B)=\PP(a_{e+\gamma}\in B)$ as well as $\PP(b_e=1)=\PP(b_{e+\gamma}=1)$ for any $e\in E$, $\gamma\in \ZZ^d$ and $B\in\cB(\RR)$. 
Furthermore, as $T_\gamma$ is a translation, $\PP(Z)=\PP(T_\gamma(Z))$ holds obviously for any $\gamma\in \ZZ^d$ and any cylinder set $Z\in \cZ$, which implies the same property for any set $B\in \cA$, c.f. Remark \ref{rem:cylinder}.

To prove ergodicity let $B\in\cA$ with $B=T_\gamma(B)$ for all $\gamma\in\ZZ^d$ and $\PP(B)>0$ be given. 
We need to show that this implies $\PP(B)=1$. 
In the following we apply the approximation lemma for measures, 
which belongs to the entourage of Carath\'eodory's extension theorem, 
cf.~e.g.{} Theorem 1.65 in \cite{Klenke-08}.
Let $\epsilon>0$. 
As $B\in\cA=\sigma(\cZ)$ and as $\cZ$ is a semiring we can find cylinder sets $Z_1,\dots,Z_n\in \cZ$ such that
\[
 \PP(B\triangle  Z)<\epsilon\quad\text{where}\quad Z:=\bigcup_{k=1}^n Z_k,
\]
which gives
\begin{equation}\label{eq1}
 \PP(B)^2 -2\PP(B)\epsilon  \leq\PP(Z)^2\leq \PP( B)^2 +2\PP(B)\epsilon +\epsilon^2.
\end{equation}
Furthermore we have for any $\gamma\in \ZZ^d$
\begin{align*}
 \PP(Z\cap T_\gamma Z)
&\leq  \PP\left(( B\cup (Z\setminus B))\cap T_\gamma Z\right)\\
&\leq  \PP\left( B\cap T_\gamma Z\right)  + \PP\left(  (Z\setminus B)\cap T_\gamma Z\right)\\
&\leq  \PP\left( B\cap (T_\gamma B\cup (T_\gamma Z\setminus T_\gamma B)  )\right)  + \epsilon\\
&\leq  \PP\left( B\cap T_\gamma B\right) + \PP\left( B \cap (T_\gamma Z\setminus T_\gamma B)  \right)  + \epsilon\\
&\leq  \PP\left( B\cap T_\gamma B\right)  + 2\epsilon
\end{align*}
By symmetry we get for all $\gamma\in \ZZ^d$
\[ \PP(B\cap T_\gamma B)-2\epsilon \leq \PP( Z\cap T_\gamma Z) \leq \PP(B\cap T_\gamma B)+2\epsilon \]
and the $T$-invariance of $B$ implies
\begin{equation}\label{eq2}
 \PP(B)-2\epsilon \leq \PP( Z\cap T_\gamma  Z) \leq \PP(B)+2\epsilon.
\end{equation}
 As $ Z$ is a finite union of cylinder sets, 
it does only depend on finitely many edges. Hence there exists an element $h\in \ZZ^d$ such that $ Z$ and $T_h Z$ are independent, which gives
\[
 \PP( Z\cap T_h Z)=\PP( Z)\PP(T_h Z) = \PP( Z)^2
\]
since $T$ is measure preserving. This gives together with (\ref{eq1}) and (\ref{eq2})
\[
 \PP(B)-2\PP(B)\epsilon-\epsilon^2-2\epsilon \leq \PP(B)^2\leq \PP(B)+ 2\PP(B)\epsilon +2\epsilon
\]
and dividing by $\PP(B)>0$ leads to
\[
 1-2\epsilon-\frac{\epsilon^2+2\epsilon}{\PP(B)} \leq \PP(B)\leq 1+ 2\epsilon + \frac{2\epsilon}{\PP(B)}
\]
As these inequalities hold for arbitrary $\PP(B)\geq\epsilon >0$ we get $\PP(B)=1$.
\end{proof}

Denote by $\ell^2(\ZZ^d)$ all square summable, complex-valued functions on $\ZZ^d$ and by $C_c(\ZZ^d)$ the subset of $\ell^2(\ZZ^d)$ 
consisting of all finitely supported functions.
Let $\alpha,\beta\in [0,1]$ be some fixed numbers. 
Using the random variables $a_e,b_e$, $e\in E$ we define for each $\omega\in \tilde\Omega$ as in Lemma \ref{lemma:locfin} the random 
operator $\tilde H(\omega):=\tilde H^{\alpha,\beta}(\omega):C_c(\ZZ^d)\to\ell^2(\ZZ^d)$ point-wise by
\[
 \tilde H_{x,y}(\omega):= \tilde H_{x,y}^{\alpha,\beta}(\omega) := \begin{cases}
           a_{\{x,y\}}(\omega) b_{\{x,y\}}(\omega) &\text{ if } x\neq y, \\
   \alpha a_{\{x\}}(\omega)b_{\{x\}}(\omega)-\beta\sum_{z\neq x} a_{\{x,z\}}(\omega)b_{\{x,z\}}(\omega) &\text{ if } x=y
          \end{cases}
 \]
and for $\phi\in C_c(\ZZ^d)$ we set
\begin{equation}\label{def:tildeH}
 (\tilde H(\omega) \phi )(x) := (\tilde H^{\alpha,\beta}(\omega) \phi )(x) :=\sum_{y\in \ZZ^d}\tilde H_{x,y}(\omega)\phi(y).
\end{equation}
It is easy to see that
\begin{equation}
\label{eq:defH2}
 (\tilde H(\omega)\phi)(x) = \sum_{\ato{y\neq x}{\{x,y\}\in E_\omega}} \left(\phi(y)- \beta\phi(x)\right) a_{\{x,y\}}(\omega) + \alpha \phi(x) a_{\{x\}}(\omega).
\end{equation}
Using this we obtain for each $\phi\in C_c(\ZZ^2)$ and $\omega\in\Omega$ such that $\Gamma_\omega$ is locally finite that $\tilde H(\omega)\phi\in \ell^1(\ZZ^d)\subseteq\ell^2(\ZZ^d) $. To see this we set $A:=\supp \phi$, $m:=\max_{x\in A}|\phi(x)|$ and $N_y(\omega):=\{x\in \ZZ^d\mid \{x,y\}\in E_\omega \}$ to estimate
\begin{align*}
 \sum_{x\in\ZZ^d}\left| \sum_{y\in \ZZ^d} \tilde H_{x,y}(\omega)\phi(y)  \right|
&\leq 
\sum_{x\in\ZZ^d} \sum_{y\in A} \left|\tilde H_{x,y}(\omega)\right| \left|\phi(y)  \right|
\leq
 m \sum_{x\in\ZZ^d} \sum_{y\in A} \left|\tilde H_{x,y}(\omega)\right| \\
&\leq
 m \sum_{y\in A} \sum_{x\in N_y(\omega)} \left|\tilde H_{x,y}(\omega)\right| 
\leq
 m \sum_{y\in A} \left(|\tilde H_{y,y}(\omega)|+\sum_{\ato{x\in N_y(\omega)}{x\neq y}} \left|\tilde H_{x,y}(\omega)\right| \right) \\
&\leq
 m \sum_{y\in A} \left(|a_{\{y\}}(\omega)|+2\sum_{\ato{x\in N_y(\omega)}{x\neq y}} \left|\tilde H_{x,y}(\omega)\right| \right)
<\infty.
\end{align*}
 Note that here we used that $N_y(\omega)$ is finite, as the underlying graph $\Gamma_\omega$ is locally finite. 
In the sense of \cite[Section $\S$.1.B]{FigotinP-92} the mapping 
\[
 \tilde H:\Omega\to L(\ell^2(\ZZ^d)),\quad  \omega\mapsto
\begin{cases}
\tilde H(\omega)& \text{if }\omega\in \tilde\Omega,\\
\Id & \text{else.}
\end{cases}
\]
is a random operator with domain $C_c(\ZZ^d)$. 
This means that almost surely $C_c(\ZZ^d)$ is a subset of the domain of $\tilde H$ and almost surely $\tilde H u$ is for all $u\in C_c(\ZZ^d)$ a random vector.  Note that here $L(\ell^2(\ZZ^d))$ is the space of the linear operators which are densely defined in $\ell^2(\ZZ^d)$.

\begin{Remark}
\label{r:alphabeta}
 The operator $\tilde H(\omega)$ depends on the choice of $\alpha,\beta\in [0,1]$ and is defined on the finitely supported functions in $\ell^2(\ZZ^d)$. In Lemma \ref{la:sa} we will define the self-adjoint extension $H(\omega)$ of this operator. Depending on $\alpha$ and $\beta$ we have in particular the following special cases for $H(\omega)$:
\begin{itemize}
 \item if $\alpha=0$ and $\beta=1$, then $H(\omega)$ is the randomly weighted Laplacian on the graph $\Gamma_\omega$,
 \item if $\alpha=\beta=1$, then $H(\omega)$ is the randomly weighted Laplacian on the graph $\Gamma_\omega$ plus a random diagonal,
 \item if $\alpha=1$ and $\beta=0$, then $H(\omega)$ is the randomly weighted adjacency operator of $\Gamma_\omega$ with a random diagonal,
 \item if $\alpha=\beta=0$, then $H(\omega)$ is the randomly weighted adjacency operator of $\Gamma_\omega$ with zeros on the diagonal.
\end{itemize}
 The diagonal elements which appear if $\alpha>0$ can be interpreted, either as random weights on the loops or as a random potential. For values $\alpha,\beta\in(0,1)$ the operator can be seen as an interpolation between, the adjacency operator and the Laplacian respectively Schr\"odinger operator of the graph $\Gamma_\omega$.
\end{Remark}

We will use the same symbol $T_\gamma$ for a mapping $T_\gamma:L(\ell^2(\ZZ^d))\to L(\ell^2(\ZZ^d))$ defined by
\[
  T_\gamma((A_{x,y})_{x,y\in\ZZ^d}):= (A_{x+\gamma,y+\gamma})_{x,y\in\ZZ^d},
\]
for arbitrary $A=(A_{x,y})_{x,y\in \ZZ^d}\in L(\ell^2(\ZZ^d))$. We set $U_\gamma:\ell^2(\ZZ^d)\to\ell^2(\ZZ^d)$ 
\[
 U_\gamma((\phi(x))_{x\in\ZZ^d}):=(\phi(x+\gamma))_{x\in\ZZ^d}
\]
where $\phi=(\phi(x))_{x\in \ZZ^d}$ is arbitrary. Then obviously $T_\gamma (A) = U_\gamma A U_\gamma^{-1}$.

For each $x,y,\gamma\in \ZZ^d$ with $x\neq y$ and $\omega= (\omega_e',\omega''_e)_{e\in E}$ we set $s:=\{x,y\}$ and have
\begin{align*}
 \tilde H_{x,y}(T_\gamma (\omega))
&=a_s (T_\gamma(\omega)) b_s (T_\gamma(\omega)) \\
&=a_s ( (\omega_{e+\gamma}',\omega_{e+\gamma}'')_{e\in E} ) b_s ((\omega_{e+\gamma}',\omega_{e+\gamma}'')_{e\in E}) \\
&=\omega_{s+\gamma}' \cdot \omega_{s+\gamma}'' \\
&=\omega_{\{x+\gamma,y+\gamma\}}' \cdot \omega_{\{x+\gamma,y+\gamma\}}'' \\
&=a_{\{x+\gamma,y+\gamma\}}(\omega)   \cdot b_{\{x+\gamma,y+\gamma\}}(\omega) = \tilde H_{x+\gamma,y+\gamma}(\omega)
\end{align*}
Furthermore we have for the diagonal elements
\begin{align*}
 \tilde H_{x,x}(T_\gamma (\omega))
&= \alpha a_{\{x\}}(T_\gamma (\omega))b_{\{x\}}(T_\gamma (\omega))-\beta\sum_{z\neq x} a_{\{x,z\}}(T_\gamma (\omega))b_{\{x,z\}}(T_\gamma (\omega))\\
&= \alpha a_{\{x+\gamma\}}(\omega)b_{\{x+\gamma\}}(\omega) -\beta\sum_{z\neq x} a_{\{x+\gamma,z+\gamma\}}(\omega)b_{\{x+\gamma,z+\gamma\}}(\omega)
=\tilde H_{x+\gamma,x+\gamma}(\omega).
\end{align*}
Therefore we have 
\begin{equation}\label{eq:transl}
 \tilde H(T_\gamma(\omega))=T_{\gamma}(\tilde H(\omega)).
\end{equation}

 \begin{Definition} 
 Let $A$ be a random operator mapping each element of the probability space $(\Omega,\cA,\PP)$ to an linear operator on the Hilbert space $\cH$. Then $A$ is called \emph{metrically transitive}, if there exists a group $\cT$ of measure preserving automorphisms of $(\Omega,\cA,\PP)$, a group of unitary operators $\cU:=\{U_T\mid T\in\cT\}$ on $\cH$ and a homomorphism from $\cT$ to $\cU$ such that
 \begin{equation}\label{eq:mt_erg}
  B\in \cA \text{ such that } TB=B \text{ for all }T\in\cT \quad \Rightarrow \quad \PP(B)\in\{0,1\}  
 \end{equation}
 and one has for all $\omega\in\Omega$ and all $T\in\cT$ the relation
\begin{equation}\label{eq:mt_transl}
 A(T\omega)= U_T A(\omega)U_T^{-1}.
\end{equation}
\end{Definition}

The next aim is to prove that $\tilde H$ is essentially selfadjoint and that $\tilde H$ and its selfadjoint extension are metrically transitive.

\begin{Lemma}\label{la:sa}
Let $(\Omega,\cA,\PP)$ and the random operator $\tilde H$ be given as above. Then
\begin{enumerate}[(a)]
 \item there exists a set $\Omega'$ of full measure such that for each $\omega\in\Omega'$ the operator $\tilde H(\omega)$ is essentially self-adjoint. We denote the closure of $\tilde H(\omega)$ by $H_\omega$ and its domain by $D(\omega)$.
 \item the random operators $\tilde H$ and 
\[
H:\Omega\to L(\ell^2(\ZZ^d))\quad\text{given by}\quad
 H(\omega)= \begin{cases}
                H_\omega& \text{ if }\omega\in \Omega' \\
                \Id & \text{ otherwise}
               \end{cases}
\]
are metrically transitive. Here $\Id$ is the identity operator in $L(\ell^2(\ZZ^d))$.
\end{enumerate}
\end{Lemma}

To prove the Lemma we will make use of the following theorem due to Figotin \cite{Figotin-87}, see also \cite{FigotinP-92}.
\begin{Theorem}[\cite{Figotin-87}, \cite{FigotinP-92}]
\label{Fig1} Let $(\Omega,\cA,\PP)$ be a probability space and $A$ a metrically transitive random operator with domain $C_c(\ZZ^d)$ satisfying
\begin{equation}\label{Fig1a}
 \EE\left( \bigg( \sum_{x\in\ZZ^d} |A_{0,x}| \bigg)^2 \right)<\infty.
\end{equation}
Then the operator $A(\omega)$ is for almost all $\omega\in\Omega$ essentially self-adjoint.
\end{Theorem}

\begin{proof}[Proof of Lemma \ref{la:sa}]
First we show that $\tilde H$ is metrically transitive. To this end, define $\cT$ and $\cU$ as follows
\[
 \cT:=\{T_\gamma\mid \gamma\in \ZZ^d\},\quad \cU:=\{U_{T_\gamma}:=U_\gamma \mid \gamma\in\ZZ^d\}.
\]
and set $\phi:\cT\to\cU$, $\phi(T_\gamma)=U_\gamma$, which clearly is a homomorphism. It is obvious that $\cT$ and $\cU$ are groups and it is easy to prove that each $U_\gamma$ is unitary. Furthermore we know from Lemma \ref{lemma:Tergodic} that the translations $T_\gamma$ are measure preserving automorphisms of the probability space. Property \eqref{eq:mt_erg} follows from the ergodicity of $T$ shown in Lemma 
\ref{lemma:Tergodic} as well. From line \eqref{eq:transl} we infer that \eqref{eq:mt_transl} holds.

In order to apply Theorem \ref{Fig1}, to show that $\tilde H$ is almost surely essentially selfadjoint, it remains to prove \eqref{Fig1a} for $\tilde H$. Therefore we consider for each $\omega\in\Omega$
\begin{align}
 \bigg(\sum_{x\in\ZZ^d}|\tilde H_{0,x}(\omega)|\bigg)^2 
&= \bigg( |H_{0,0}(\omega)| + \sum_{x\neq 0} |H_{0,x}(\omega)|  \bigg)^2 \nonumber \\
&\leq \bigg(\alpha |a_{\{0\}}(\omega)b_{\{0\}}(\omega)| + (\beta+1)\sum_{x\neq 0}|a_{\{0,x\}}(\omega)b_{\{0,x\}}(\omega)| \bigg)^2     \label{eq:la31a} \\
&\leq 4\bigg(\sum_{x\in\ZZ^d} |a_{\{0,x\}}(\omega)|b_{\{0,x\}}(\omega) \bigg)^2. \nonumber
\end{align}
For each $\omega\in\Omega$ set $N(\omega):=\{x\in\ZZ^d \mid b_{\{0,x\}}(\omega)=1 \}$. By Lemma \ref{lemma:locfin} there exists a set $\tilde \Omega\subseteq  \Omega$ of full measure such that $|N(\omega)|<\infty$ for all $\omega\in\tilde\Omega$. For $\omega \in\tilde \Omega$ we have
\[
  \Big(\sum_{x\in\ZZ^d} |a_{\{0,x\}}(\omega)|b_{\{0,x\}}(\omega)\Big)^2 
= \Big(\sum_{x\in N(\omega)}|a_{\{0,x\}}(\omega)|\Big)^2
\leq |N(\omega)|\cdot \sum_{x\in N(\omega)}|a_{\{0,x\}}(\omega)|^2.
\]
Taking the expectation value on both sides and the application of the monotone convergence theorem leads to
\begin{align*}
 \EE\bigg(\Big(\sum_{x\in\ZZ^d} |a_{\{0,x\}}(\omega)|b_{\{0,x\}}(\omega) \Big)^2\bigg)
\leq
 \EE\bigg(\Big( \sum_{x\in\ZZ^d} |a_{\{0,x\}}|^2 b_{\{0,x\}} |N| \Big)\bigg)
\leq
 v^2 \sum_{x\in\ZZ^d} \EE(b_{\{0,x\}}|N|),
\end{align*}
where $v^2$ is the upper bound for the second moments given in \eqref{eq:def:mu}.
For each $x\in \ZZ^d$ and $\omega\in\Omega$ we set $N_x(\omega):=|N(\omega)\setminus\{x\}|$, then we obtain for fixed $x\in\ZZ^d$
\[
 \EE( b_{\{0,x\}} |N|)
= \sum_{k=1}^\infty k \cdot \PP(b_{\{0,x\}}(\omega)=1, N_x(\omega)=k-1)
= \PP(b_{\{0,x\}}(\omega)=1) \EE(N_x+1).
\]
Using $\EE(N_x)\leq \EE(|N|)=\|p\|_1<\infty$ this implies
\[
 \EE\bigg(\Big(\sum_{x\in\ZZ^d} |a_{\{0,x\}}(\omega)|b_{\{0,x\}}(\omega) \Big)^2\bigg)
\leq v^2 (\EE(|N|)+1) \sum_{x\in\ZZ^d} p(x) = v^2(\|p\|_1^2 +\|p\|_1)
<\infty.
\]
 This shows together with \eqref{eq:la31a} the finiteness of the expression in \eqref{Fig1a} for the operator $\tilde H$. Hence Theorem \ref{Fig1} gives a set $\Omega'$ of full measure such that for each $\omega\in\Omega'$ the operator $\tilde H(\omega)$ is essentially selfadjoint. This proves of part (a).

To complete the prove of part (b) it remains to show that the operator $H$ is metrically transitive. 
This follows by the same argument as we used to prove metrically transitivity of $\tilde H$. Note that here we use that \eqref{eq:transl} hold for $H$ as well.
\end{proof}

The operator $H$ defined as in Lemma \ref{la:sa} is a random operator with domain $C_c(\ZZ^d)$, c.f. \cite{FigotinP-92}. We will refer to this operator as \emph{weighted Hamiltonian on the graph $\Gamma_\omega$}.

Let $(\Lambda_n)$ be a sequence of cubes given by
\begin{equation}\label{def:cube}
\Lambda_n:=([-n,n]\cap\ZZ)^d \quad (n\in\NN) 
\end{equation}
 and  for each $n\in \NN$ let $H_n(\omega)$ be the restriction of $H(\omega)$ to $\Lambda_n$. To be precise, for $\Lambda\subseteq\ZZ^d$ let the inclusion $i_\Lambda:\ell^2(\Lambda)\to\ell^2(\ZZ^d)$ and the projection $p_\Lambda:\ell^2(\ZZ^d)\to\ell^2(\Lambda)$ be given by
\[
 (i_\Lambda \phi)(x)=\begin{cases}\phi(x)& \text{if }x\in \Lambda\\ 0&\text{otherwise}\end{cases}
\quad\text{and}\quad
 (p_\Lambda \psi)(y)=\psi(y)
\]
for all $x\in\ZZ^d, y\in \Lambda, \phi\in\ell^2(\Lambda)$ and $\psi\in\ell^2(\ZZ^d)$. Then we set 
\[
H_n(\omega):= p_{\Lambda_n} H(\omega) i_{\Lambda_n}:\ell^2(\Lambda_n)\to\ell^2(\Lambda_n)
\]
 For each $\omega\in\Omega$ and $n\in\NN$ we define a function $F_n^{\omega}:\RR\to \RR$ by 
\begin{align}\label{def:F_n}
 F_n^{\omega}(\lambda) := \left|\{ \text{ eigenvalues of }H_n(\omega) \text{ not exceeding }\lambda  \}\right|,
\end{align}
where we count the eigenvalues with their multiplicity.
Therefore $F_n^{\omega}$ is the distribution function of a measure which we will denote by $\rho_n^{(\omega)}$.
Note that $|\Lambda_n|^{-1}\rho_n^{(\omega)}$ is a probability measure.

\section{Weak convergence}
\label{s:Weak}
In order to prove weak convergence of the approximating distribution functions we make use of an abstract result by Figotin \cite{Figotin-87}, see also Theorem 4.8 in \cite{FigotinP-92}, which we now cite in a special case.

\begin{Theorem} [\cite{Figotin-87}, \cite{FigotinP-92}]  \label{Fig2}
Let $(\Omega,\cA,\PP)$ be a probability space and $A$ a metrically transitive random operator with domain $C_c(\ZZ^d)$ such that
\[
\sum_{x\in\ZZ^d} \EE(|A_{0,x}|)<\infty
\]
and assume that $A$ is almost surely self-adjoint.
Then there exists a set $\tilde\Omega\subseteq \Omega$ of full measure such that for all $\omega\in\tilde\Omega$ and all $\lambda\in \{s\in\RR\mid F \text{ is continuous in }s\}$ one has
\[
 \lim_{n\to\infty} \frac{F_n^\omega(\lambda)}{|\Lambda_n|}= F(\lambda)
\]
where the limit $F:\RR\to[0,1]$ given by $\lambda\mapsto \EE\{\sprod{E_A( (-\infty,\lambda] )\delta_0}{\delta_0} \}$ is a distribution function of a probability measure.
Note that here $E_{A(\omega)}((-\infty,\lambda])$ is the spectral projection in the interval $(-\infty,\lambda]$ of the operator $A(\omega)$ and $\delta_x\in \ell^2(\ZZ^d)$ denotes the element with $\delta_x(y)=1$ if $x=y$ and $\delta_x(y)=0$ otherwise.
\end{Theorem}
This theorem and the previous considerations immediately give the following theorem.

\begin{Theorem}\label{theorem:weak}
 Let the probability space $(\Omega,\cA,\PP)$ and the operator $H$ be the weighted Hamiltonian given in Section 2. Set
$$F:\RR\to [0,1],\quad F(\lambda):=\EE\{\sprod{E_H( (-\infty,\lambda] )\delta_0}{\delta_0} \}.$$
Then there exists a set $\tilde \Omega\subseteq \Omega$ of full measure, such that for all $\omega \in\tilde \Omega$ the distribution functions $F_n^\omega/|\Lambda_n|$ converge to the distribution function $F$ point-wise at all points of continuity of $F$.
\end{Theorem}
\begin{proof}
By definition the operator $H(\omega)$ is self-adjoint for all $\omega\in\Omega$. Furthermore $H$ has domain $C_c(\ZZ^d)$ and $H$ is metrically transitive by Lemma \ref{la:sa}. The finiteness of $\sum_{x\in\ZZ^d} \EE(|H_{0,x}|)$ follows from
\begin{align*}
 \sum_{x\in\ZZ^d} \EE(| H_{0,x}|)
&\leq \alpha \EE(|a_{\{0\}}|b_{\{0\}}) + \beta \EE\bigg(\sum_{x\neq 0}|a_{\{0,x\}}|b_{\{0,x\}}\bigg) +\sum_{x\neq 0} \EE(|a_{\{0,x\}}|b_{\{0,x\}}) \\
&\leq 2 \sum_{x\in \ZZ^d}\EE(|a_{\{0,x\}}|b_{\{0,x\}}) \leq 2(v^2+1) \|p\|_1 <\infty,
\end{align*}
where $v^2$ is the upper bound for the second moments given in \eqref{eq:def:mu}. Hence, Theorem \ref{Fig2} implies the claim of the theorem.
\end{proof}

\section{Control of the jumps}
\label{s:Jumps}

The aim of this section is to control the jumps of the limit function given in Corollary \ref{theorem:weak} in order to obtain uniform convergence of the approximants.
In the following we will make use of Birkhoff's ergodic theorem in the $d$-dimensional case, see \cite{Keller-98}.
\begin{Theorem}\label{theorem:linde}
  Let $\ZZ^d$ act from the left on a probability space $(\Omega,\cA,\PP)$ by an ergodic and measure preserving transformation $T$ an let $(\Lambda_n)$ be the sequence of cubes given as in \eqref{def:cube}. Then for any $ f\in L^1(\PP)$
  \[
  \lim_{n\rightarrow\infty}\frac{1}{\vert \Lambda_n \vert}\sum_{g\in \Lambda_n}f(T_g \omega ) =\int f(\omega) d \PP(\omega)
  \]
  holds almost surely.
\end{Theorem}

\begin{Lemma}\label{lemma:LV2}
Let $(\Omega,\cA,\PP)$ be the probability space, $H$ the randomly weighted Hamiltonian and $(\Lambda_n)$ be the sequence of cubes given as in \eqref{def:cube}. Then there exists a set $\tilde\Omega\subseteq\Omega$ of full measure such that for all $\omega\in\tilde\Omega$ and all $\lambda\in\RR$ we have
\[
 \lim_{n\to\infty}\frac{\Tr(\chi_{\Lambda_n}E_{H(\omega)}(\{\lambda\}))}{|\Lambda_n|}=\EE\{ \sprod{E_{H(\omega)}(\{\lambda\})\delta_0}{ \delta_0} \}.
\]
\end{Lemma}
\begin{proof}
 Let $\omega\in\Omega$ be fixed. By definition of the trace we have
\begin{align}\label{la:jumpA1}
 \Tr(\chi_{\Lambda_n}E_{H(\omega)}(\{\lambda\}))
=\sum_{x\in \ZZ^d} \sprod{\chi_{\Lambda_n}E_{H(\omega)}(\{\lambda\})\delta_x }{\delta_x}
=\sum_{x\in \Lambda_n} \sprod{E_{H(\omega)}(\{\lambda\})\delta_x}{\delta_x}.
\end{align}
Let $\Eig (H(\omega),\lambda)$ denote the eigenspace of $H(\omega)$ corresponding to the value $\lambda$, which could possibly be empty if $\lambda$ is not an eigenvalue. Given $\gamma\in \ZZ^d$, we have $\phi\in \Eig (H(\omega),\lambda)$ if and only if $T_\gamma(\phi)\in \Eig(T_\gamma(H(\omega)),\lambda)$. 

Using this we prove
\begin{equation}\label{la:jumpA2}
 \sprod{E_{T_z(H(\omega))}(\{\lambda\}) \delta_0}{ \delta_0} = \sprod{E_{H(\omega)}(\{\lambda\}) \delta_z}{\delta_z}.
\end{equation}
Therefore let $\delta_0'\in \Eig (T_z(H(\omega)),\lambda)$ and $\delta_0''\in \Eig (T_z(H(\omega)),\lambda)^\perp$ such that $\delta_0=\delta_0'+\delta_0''$. Then we obtain
\begin{align*}
 \sprod{E_{T_z(H(\omega))}(\{\lambda\}) \delta_0}{\delta_0} 
=\sprod{E_{T_z(H(\omega))}(\{\lambda\})  \delta_0'}{\delta_0} + \sprod{E_{T_z(H(\omega))}(\{\lambda\}) \delta_0''}{ \delta_0}  
=\sprod{\delta_0'}{\delta_0} 
\end{align*}
and with the above equivalence we get 
\begin{align*}
\sprod{\delta_0'}{\delta_0}
&= \sprod{T_{-z}(\delta_0')}{T_{-z}(\delta_0)}\\
&= \sprod{E_{H(\omega)}(\{\lambda \})T_{-z}(\delta_0')}{ T_{-z}(\delta_0)}+\sprod{E_{H(\omega)}(\{\lambda \})T_{-z}(\delta_0'')}{ T_{-z}(\delta_0)}\\
&= \sprod{E_{H(\omega)}(\{\lambda \})\delta_z}{ \delta_z},
\end{align*}
which implies \eqref{la:jumpA2}.
Applying \eqref{la:jumpA1}, \eqref{la:jumpA2} and the fact $T_x(H(\omega))=H(T_x(\omega))$ leads to 
\begin{align*}
\frac{\Tr(\chi_{\Lambda_n}E_{H(\omega)}(\{\lambda\})) }{|\Lambda_n|}
=\frac{1}{|\Lambda_n|} \sum_{x\in \Lambda_n} \sprod{E_{H(T_x(\omega))}(\{\lambda\}) \delta_0}{ \delta_0} .
\end{align*}
Finally we use Lemma \ref{lemma:Tergodic} and Theorem \ref{theorem:linde} to obtain the existence of a set $\tilde\Omega\subseteq \Omega$ of measure one such that for each $\omega\in\tilde\Omega$ we have
\begin{align*}
 \lim_{n\to\infty}\frac{\Tr(\chi_{\Lambda_n}E_{H(\omega)}(\{\lambda\})) }{|\Lambda_n|}
&= \int_\Omega \sprod{E_{H(\omega)}(\{\lambda\}) \delta_0}{ \delta_0} d\PP(\omega)
\end{align*}
which was to prove.
\end{proof}
The following fact is taken from \cite{LenzV-09}
\begin{Lemma}\label{lemma:LV1}
 Let $r>0$, $\Lambda\subseteq \ZZ^d$ and $U\subseteq \ell^2(\Lambda)$ be given and denote by $U_r$ the subspace of $U$ consisting of all functions which vanish on $\partial^r (\Lambda)$. Then
\[
 0\leq \dim(U)-\dim (U_r) \leq |\partial^r(\Lambda)|.
\]
\end{Lemma}
\begin{proof}
 Let $P:U\to\ell^2(\partial^r(\Lambda))$ be the natural projection with $(P\phi)(x)=\phi(x)$ for all $x\in \partial^r(\Lambda)$. Then we have
\[
 0\leq \dim(U)-\dim(\ker P) =\dim (\ran P) \leq |\partial^r (\Lambda)|,
\]
which proves the claim as $\ker P= U_r$.
\end{proof}
For given $\omega\in\Omega$, $R\in\NN$ and $Q\subseteq \ZZ^d$ finite, let $L^{(\omega)}(R,Q)$ 
denote the number of $e \in E$ with $b_e(\omega)=1$ which are of length not less than $R$ and incident to some vertex in $Q$, i.e.
\begin{equation}\label{def:Lomega}
 L^{(\omega)}(R,Q):=\left|\left\{\{x,y\} \in E \mid b_{\{x,y\}}(\omega)=1, d(x,y)\geq R \text{ and } \{x,y\}\cap Q\neq \emptyset \right\}\right|.
\end{equation}

 Let $(\Lambda_n)$ be the sequence of cubes given as in \eqref{def:cube}. We chose a function $R:\NN\to\NN$ such that 
\begin{equation}\label{def:R}
 \lim_{n\to\infty} R(n)=\infty \quad\text{and}\quad \lim_{n\to\infty}\frac{|\partial^{R(n)}\Lambda_n|}{|\Lambda_n|}=0
\end{equation}
and set
\begin{equation}\label{def:Lomega_spez}
 L^{(\omega)}_n:=L^{(\omega)}(R(n),\Lambda_n).
\end{equation}
Beside this we set for $R\geq 0$
\[
 \epsilon_R:=\sum_{{x\in \ZZ^d, d(0,x)\geq R}} p(x) 
\]
and for $n\in\NN_0$
\begin{equation}\label{def:ep}
 \epsilon(n):= \epsilon_{R(n)}\quad\text{as well as}\quad\delta(n):= (2n+1)^{-d/4}.
\end{equation}
Note as $p\in\ell^1(\ZZ^d)$ we have by the definition of $R(n)$ that $$\lim_{n\to\infty}\epsilon(n)=\lim_{n\to\infty}\delta(n)=0.$$

The next result estimates the probability that the number of long edges is large.
\begin{Lemma}\label{lemma:schw}
Let $(\Omega,\cA,\PP)$ be given as above. Then the following holds:
\begin{itemize}
 \item [(a)]
There exist constants $R_0\in\NN$ and $\bar\delta>0$ such that for all $0<\delta<\bar\delta$, all $R\geq R_0$ and all finite $Q\subseteq \ZZ^d$
\[
 \PP\left(L^{(\omega)}(R,Q)\geq |Q|(\epsilon_R+\delta)\right) \leq \exp\left( -\frac{\delta^2 |Q|}{4} \right).
\]
\item [(b)]
Let $R:\NN\to\NN$ be as in \eqref{def:R} and $L_n^{(\omega)}=L^{(\omega)}(R(n),\Lambda_{n})$.
Then there exists a set $\tilde \Omega\subseteq\Omega$ of full measure such that for each $\omega\in\tilde\Omega$ there exists $n_0(\omega)$ with
\[
 L_n^{(\omega)}\leq |\Lambda_n| (\epsilon(n)+\delta(n))\quad\quad(n\geq n_0(\omega)).
\]
 \end{itemize}
\end{Lemma}
\begin{proof}
 The proof of part (a) is to be found in \cite{Schwarzenberger-12}. It is basically an application of a Bernstein inequality. Let us prove part (b). Therefore consider the events
\[
 A_n:=\left\{\omega\in\Omega\mid L_n^{(\omega)}\geq |\Lambda_n|(\epsilon(n)+\delta(n))\right\}.
\]
Then part (a) shows that for $n$ large enough we have 
\[
\PP(A_n)\leq \exp\left(-\delta(n)^2|\Lambda_n|/4\right) =\exp\bigl(-(2n+1)^{d/2}/4\bigr), 
\]
which clearly gives $\sum_{n\in\NN}\PP(A_n)<\infty$. By the Lemma of Borel Cantelli we have 
\[
 \PP\Bigl(\limsup_{n\to\infty} A_n\Bigr)=0
\]
which implies the claim of part (b).
\end{proof}

We use Lemmas \ref{lemma:LV2}, \ref{lemma:LV1} and \ref{lemma:schw} to obtain a result similar to Lemma 6.2 in \cite{LenzV-09}
\begin{Theorem}\label{theorem:technical}
Let $(\Omega,\cA,\PP)$ be the probability space, $H$ the randomly weighted Hamiltonian, $(\Lambda_n)$ be the sequence of cubes and $\rho_n^{(\omega)}$ the measures associated to the eigenvalue counting functions given as before. Then there exists a set $\tilde\Omega\subseteq\Omega$ of full measure such that for all $\omega\in\tilde\Omega$ and all $\lambda\in\RR$ we have
\[
\lim_{n\to\infty}\frac{\rho_{n}^{(\omega)}(\{\lambda\}) }{|\Lambda_n|}=\EE\{ \sprod{E_H(\{\lambda\})\delta_0}{\delta_0} \}.
\]
\end{Theorem}
\begin{proof}
 Let $\tilde\Omega\subseteq\Omega$ be a set of full measure such that the results of Lemma \ref{lemma:LV2} and of Lemma \ref{lemma:schw} (b) hold for all $\omega\in\tilde\Omega$. We fix some $\omega\in\tilde\Omega$ and $\lambda\in \RR$. With the function $R:\NN\to\NN$ given in \eqref{def:R} we set
\[
 V_n^{(\omega)}:=\left\{ v\in\ell^2(\ZZ^d)\mid (H(\omega)-\lambda)v=0 \text{ and }\supp v \subseteq \Lambda_{n-R(n)}\right\},\quad D_n^{(\omega)}:= \dim V_n^{(\omega)}.
\]
Note that $V_n^{(\omega)}$ consists of the elements $i_{\Lambda_{n}}v$, where $v\in \ell^2(\Lambda_{n})$ satisfying $v\equiv 0$ on $\Lambda_n\setminus \Lambda_{n-R(n)}$,
\begin{equation}\label{eq:la62a}
(p_{\Lambda_n}H(\omega)i_{\Lambda_n}-\lambda) v=0
\quad\text{ and }\quad
  \sum_{y\in\Lambda_{n-R(n)}}(H_{x,y}(\omega)-\lambda\delta_{x}(y))v(y)=0 
\end{equation}
 for all $x\in \Lambda_n^{\rm c}$ with $x\stackrel{\omega}{\sim} \Lambda_{n-R(n)}$.

We consider the following difference
\begin{align}\label{eq:la62e}
 |\rho_n^{(\omega)}(\{\lambda\})- \Tr(\chi_{\Lambda_n}E_{H(\omega)})|
\leq |\rho_n^{(\omega)}(\{\lambda\})-D_n^{(\omega)}| + |D_n^{(\omega)}-\Tr(\chi_{\Lambda_n}E_{H(\omega)})|
\end{align}
and treat the two summands on the right hand side separately. Let us estimate the first one. Consider therefore the sets
\[
 U_n^{(\omega)}:=\left\{ u\in \ell^2(\Lambda_n)\mid (p_{\Lambda_n}H(\omega) i_{\Lambda_n}- \lambda)u=0 \right\}
\]
and
\[
 U_{n,R}^{(\omega)}=\left\{ u\in U_n \mid u\equiv 0 \text{ on } \Lambda_n\setminus\Lambda_{n-R(n)} \right\}.
\]
Then clearly, $\rho_n^{(\omega)}(\{\lambda\})= \dim(U_n^{(\omega)})\geq \dim (V_n^{(\omega)})$ and 
\begin{equation}
\label{eq:Lnomega}
 \dim(U_{n,R}^{(\omega)})-\dim(V_n^{(\omega)})\leq  |\{ y\in \Lambda_n^{\rm c}\mid y \stackrel{\omega}{\sim} \Lambda_{n-R(n)} \}|
\leq L^{(\omega)}(R(n),\Lambda_{n})= L_n^{(\omega)},
\end{equation}
where we used the definition \eqref{def:Lomega}.
The application of Lemma \ref{lemma:LV1} gives
\begin{align}
 0\leq \rho_n^{(\omega)}(\{\lambda \})-D_n^{(\omega)}
 =   \dim(U_n^{(\omega)})-\dim(V_n^{(\omega)}) 
&\leq \dim (U_n^{(\omega)})-\dim(U_{n,R}^{(\omega)})+L_n^{(\omega)} \nonumber \\
&\leq |\partial^{R(n)}\Lambda_n|+L_n^{(\omega)}. \label{eq:la62f}
\end{align}
Now we estimate the second summand in \eqref{eq:la62e}. Therefore let $v_1,\dots,v_{D_n^{(\omega)}}$ be an orthonormal basis (ONB) of $V_n^{(\omega)}$ and let $\tilde v_i$, $i\in I$ be an ONB of the orthogonal complement of $V_n^{(\omega)}$ in the space $\Eig(H(\omega),\lambda)$. 
Furthermore let $\bar v_j$, $j\in J$ be an ONB of $\Eig(H(\omega),\lambda)^\perp$. Then we have
\begin{align*}
 \Tr(\chi_{\Lambda_n} E_\omega(\{\lambda\})) 
&=\sum_{i=1}^{D_n^{(\omega)}} \sprod{\chi_{\Lambda_n} E_\omega(\{\lambda\})v_i}{v_i}+\sum_{i\in I} \sprod{\chi_{\Lambda_n} E_\omega(\{\lambda\})\tilde v_i}{\tilde v_i}+\sum_{i\in J} \sprod{\chi_{\Lambda_n} E_\omega(\{\lambda\})\bar v_i}{\bar v_i}\\
&=\sum_{i=1}^{D_n^{(\omega)}} \sprod{v_i}{v_i}+\sum_{i\in I} \sprod{\chi_{\Lambda_n} \tilde v_i}{\chi_{\Lambda_n} \tilde v_i}
\end{align*}
which gives $D_n^{(\omega)}\leq \Tr(\chi_{\Lambda_n} E_\omega(\{\lambda\})) $. Now let $u_i$, $i\in I$ be an ONB of $$\bar U_n^{(\omega)}:=\ran (\chi_{\Lambda_n}E_\omega(\{\lambda\}))$$ and $\tilde u_j$, $j\in J$ be an ONB of $(\bar U_n^{(\omega)})^\perp$. Then, using Cauchy Schwarz inequality, we obtain 
\[
 \sprod{\chi_{\Lambda_n} E_\omega(\{\lambda\})u_i}{u_i}\leq \|\chi_{\Lambda_n} E_\omega(\{\lambda\})u_i \|  \|u_i\|\leq 1 \quad\text{and}\quad \sprod{\chi_{\Lambda_n} E_\omega(\{\lambda\})\tilde u_j}{\tilde u_j}=0
\]
for all $i\in I$ and all $j\in J$. This gives
\begin{align}\label{eq:la62b}
 D_n^{(\omega)}
\leq \Tr(\chi_{\Lambda_n} E_\omega(\{\lambda\}))
=\sum_{i\in I}\sprod{\chi_{\Lambda_n} E_\omega(\{\lambda\})u_i}{u_i}+\sum_{j\in J}\sprod{\chi_{\Lambda_n} E_\omega(\{\lambda\})\tilde u_j}{\tilde u_j}
\leq \dim(\bar U_n^{(\omega)}).
\end{align}
where we used $\dim(\bar U_n)=|I|$. As before we denote by $\bar U_{n,R}^{(\omega)}$ the subset of $\bar U_n^{(\omega)}$ consisting of the elements which vanish outside of $\Lambda_{n-R}$. Therefore we have
\begin{align}\label{la:jumpA3}
 \bar U_{n,R}^{(\omega)}=\left\{ \chi_{\Lambda_n}v\mid v\in \ell^2(\ZZ^d),  (H(\omega)-\lambda)v=0, v\equiv 0 \text{ on } \partial^{R(n)}\Lambda_n \right\}.
\end{align}
In the next step we define a set $\bar{\bar U}_{n,R}^{(\omega)}\supseteq \bar U_{n,R}^{(\omega)}$ by dropping conditions in \eqref{la:jumpA3}, in the following way
\begin{align*}
 \bar{\bar U}_{n,R}^{(\omega)}
&:=\!\!\left\{\!\! \chi_{\Lambda_n}v \Bigg| v\in \ell^2(\ZZ^d), \!\sum_{y\in \ZZ^d}(H_{x,y}(\omega)-\lambda \delta_x(y))v(y)=0 \text{ for all } x \in Z_n^{(\omega)}, v\equiv 0 \text{ on } \partial^{R(n)}\Lambda_n\!\!\right\}\\
&=\!\!\left\{\!\! \chi_{\Lambda_n}v \Bigg| v\in \ell^2(\ZZ^d), \!\sum_{y\in \Lambda_{n}}(H_{x,y}(\omega)-\lambda \delta_x(y))v(y)=0 \text{ for all } x \in Z_n^{(\omega)}, v\equiv 0 \text{ on } \partial^{R(n)}\Lambda_n\!\!\right\}\!\!,
\end{align*}
where 
\[
Z_n^{(\omega)}=\Lambda_{n-R(n)}\setminus \{x\in \Lambda_{n-R(n)} \mid x\stackrel{\omega}{\sim}\Lambda_n^{\rm c}\}. 
\]
Here we used that for all $x\in Z_n^{(\omega)}$ and $y\in \Lambda_n^{\rm c}$ we have $H_{x,y}(\omega)=0$.

Comparing this representation of $\bar{\bar U}_{n,R}^{(\omega)}$ with the representation $V_n^{(\omega)}$ in \eqref{eq:la62a}, we realize that they differ in at most $2 L_n^{(\omega)} + |\partial^{R(n)}\Lambda_n|$ conditions. As each of these conditions may change the dimension at most by one, we get
\begin{equation}\label{eq:la62c}
\dim (\bar U_{n,R}^{(\omega)})
\leq\dim (\bar{\bar U}_{n,R}^{(\omega)}) 
\leq  D_n^{(\omega)} + 2 L_n^{(\omega)} + |\partial^{R(n)}\Lambda_n|.
\end{equation}
Applying \eqref{eq:la62b}, Lemma \ref{lemma:LV1} and \eqref{eq:la62c} gives
\begin{align}\label{eq:la62d}
0\leq  \Tr(\chi_{\Lambda_n} E_\omega(\{\lambda\}))  - D_n^{(\omega)}
 \leq  \dim(\bar U_n^{(\omega)})- D_n^{(\omega)}
 &\leq  \dim(\bar U_{n,R}^{(\omega)})- D_n^{(\omega)} + |\partial^{R(n)}\Lambda_n| \nonumber \\
 &\leq  2|\partial^{R(n)}\Lambda_n|+ 2 L_n^{(\omega)}
\end{align}
In the last step we apply Lemma \ref{lemma:LV2}, then we combine the estimates for the two summands in \eqref{eq:la62e} given in \eqref{eq:la62f} and \eqref{eq:la62d} and finally use part (b) of Lemma \ref{lemma:schw} to obtain
\begin{align*}
 \lim_{n\to\infty}\frac{\rho_n^{(\omega)}(\{\lambda\})}{|\Lambda_n|}-\EE(\sprod{E_{H(\omega)}(\{\lambda\})\delta_0}{\delta_0})
&=\lim_{n\to\infty} \frac{|\rho_n^{(\omega)}(\{\lambda\})- \Tr(\chi_{\Lambda_n}E_{H(\omega)})|}{|\Lambda_n|} \\
&\leq \lim_{n\to\infty} \frac{3|\partial^{R(n)}\Lambda_n|+3L_n^{(\omega)}}{|\Lambda_n|}\\
&\leq 3\lim_{n\to\infty} \left(\frac{|\partial^{R(n)}\Lambda_n|}{|\Lambda_n|}+\epsilon(n)+\delta(n)\right) =0.
\end{align*}
Here we used the definitions of $R(n)$, $\epsilon(n)$ and $\delta(n)$ in \eqref{def:R} and \eqref{def:ep}.
\end{proof}
\begin{Remark}\begin{itemize}
\item [(a)] Let us stress the fact that proof of Theorem \ref{theorem:technical} does not contain any probabilistic argument. We show the claimed convergence for any fixed choice of $\lambda\in\RR$ and $\omega\in\tilde\Omega$, 
where $\tilde \Omega$ is a  set given rather explicitly by Lemmas \ref{lemma:LV2} and \ref{lemma:schw}.
\item [(b)] Furthermore the proof gives an explicit error-term on finite scales. 
To be precise we have for any $n\in\NN$, $\lambda\in\RR$ and $\omega\in\tilde\Omega$
\[
 |\rho_n^{(\omega)}(\{\lambda\})- \Tr(\chi_{\Lambda_n}E_{H(\omega)})|\leq 3|\partial^{R(n)}\Lambda_n|+3L_n^{(\omega)}
\]
where $L_n^{(\omega)}= L^{(\omega)}(R(n),\Lambda_{n})$ as in \eqref{eq:Lnomega}.
\end{itemize}
\end{Remark}

The following result is essentially standard and has been used in the present context already in \cite{LenzV-09}. It shows that weak convergence of measures plus convergence of the measures at each point implies uniform convergence.

\begin{Lemma}\label{lemma:LV4}
 Let $\rho$ be a probability measure on $\RR$ and let $(\rho_n)$ be as sequence of bounded measures on $\RR$ which weakly converge to $\rho$ and fulfill
\[
  \lim_{n\to\infty}\rho_n(\{\lambda\}) = \rho(\{\lambda\})
\]
for all $\lambda \in \RR$. Then the distribution functions $F_n:\RR\to\RR$, $F_n(\lambda):=\rho_n((-\infty,\lambda])$ converge with respect to supremum norm to the distribution function $F:\RR\to\RR$, $F(\lambda):=\rho((-\infty,\lambda])$.
\end{Lemma}

The proof of the main theorem, already stated in Section \ref{s:Result}
is now basically a combination of the previous lemmas.
\begin{proof}[Proof of Theorem \ref{theorem:main}]
 Let $\rho,\rho_n^{(\omega)}:\cB(\RR)\to[0,1]$ be the measures associated to the distribution functions $F$ respectively $F_n^\omega$. Then obviously $\rho$ is a probability measure and the measures $\rho_n^{(\omega)}$ are bounded. As shown in Corollary \ref{theorem:weak}, there exists a set $\Omega_1\subseteq\Omega$ with $\PP(\Omega_1)=1$ such that for all $\omega\in\Omega_1$ the measure $\rho$ is the weak limit of $\rho_n^{(\omega)}$. Furthermore we have by Theorem \ref{theorem:technical} a set $\Omega_2\subseteq\Omega$ with $\PP(\Omega_2)=1$ such that for all $\omega\in\Omega_2$ and all $\lambda\in\RR$ one has $\rho_n^{(\omega)}(\{\lambda\})\to \rho(\{\lambda\})$. Therefore Lemma \ref{lemma:LV4} yields the uniform convergence of the distribution functions for all $\omega\in\Omega_1\cap\Omega_2$.
\end{proof}

\subsection*{Acknowledgment}
The authors thank Christoph Schumacher for fruitful discussions.


\end{document}